\documentclass[a4paper, 12pt]{article}
\usepackage{mathrsfs}
\usepackage{latexsym}
\usepackage{amssymb}
\usepackage{amsmath}
\usepackage{amsopn}
\usepackage{amsthm}
\usepackage{geometry}

\newtheorem{theorem}{Theorem}[section]
\newtheorem{proposition}[theorem]{Proposition}
\newtheorem{corollary}[theorem]{Corollary}
\newtheorem{lemma}[theorem]{Lemma}

\begin{document}

\title{Classification of modules over laterally complete regular algebras}

\author{Vladimir I. Chilin, Jasurbek A. Karimov}

\date{}

\maketitle

\begin{abstract}
Let $\mathcal{A}$ be a laterally complete commutative regular algebra and $X$ be a laterally complete $\mathcal{A}$-module. In this paper we introduce a notion of passport $\Gamma(X)$ for $X$, which consist of uniquely defined partition of unity in the Boolean algebra of idempotents in $\mathcal{A}$ and the set of pairwise different cardinal numbers. It is proved that $\mathcal{A}$-modules $X$ and $Y$ are isomorphic if and only if $\Gamma(X) = \Gamma(Y)$.
\end{abstract}

\emph{Key words:} Commutative regular algebra, Homogeneous module, Finite dimensional module

\emph{Mathematics Subject Classification (2000):} MSC 13C05, MSC 16D70

\section{Introduction}
\label{intro}
J.~Kaplansky \cite{lit8} introduced a class of $AW^\ast$-algebras to describe $C^\ast$-algebras, which is close to von Neumann algebras by their algebraic and order structure. The class of $AW^\ast$-algebras became a subject of many researches in the operator theory (see review in \cite{lit7}). One of the important results in this direction is the realization of an arbitrary $AW^\ast$-algebra $M$ of type $I$ as a $\ast$-algebra of all linear bounded operators, which act in a special Banach module over the center $Z(M)$ of the algebra $M$ \cite{lit9}. The Banach $Z(M)$-valued norm in this module is generated by the scalar product with values in the commutative $AW^\ast$-algebra $Z(M)$. Later, these modules were called Kaplansky-Hilbert modules (KHM). Detailed exposition of many useful properties of KHM is given, for example, in (\cite{lit5}, 7.4). One of the important properties is a representation of an arbitrary Kaplansky-Hilbert module as a direct sum of homogeneous KHM (\cite{lit10}, \cite{lit5}, 7.4.7).

Development of the noncommutative integration theory stimulated an interest to the different classes of algebras of unbounded operators, in particular, to the $\ast$-algebras $LS(M)$ of locally measurable operators, affiliated with von Neumann algebras or $AW^\ast$-algebras $M$. If $M$ is a von Neumann algebra, then the center $Z(LS(M))$ in the algebra $LS(M)$ identifies with the algebra $L^0(\Omega, \Sigma, \mu)$ of all classes of equal almost everywhere measurable complex functions, defined on some measurable space $(\Omega, \Sigma, \mu)$ with a complete locally finite measure $\mu$ (\cite{lit6}, 2.1, 2.2). If $M$ is an $AW^\ast$-algebra, then $Z(LS(M))$ is an extended $f$-algebra $C_\infty(Q)$, where $Q$ is the Stone compact coresponding to the Boolean algebra of central projectors in $M$ \cite{lit7}. The problem (like the one in the work of J.~Kaplansky \cite{lit9} for $AW^\ast$-algebras) on possibility of realization of $\ast$-algebras $LS(M)$, in the case, when $M$ has the type $I$, as $\ast$-algebras of linear $L^0(\Omega, \Sigma, \mu)$-bounded (respectively, $C_\infty(Q)$-bounded) operators, which act in corresponding KHM over the $L^0(\Omega, \Sigma, \mu)$ or over the $C_\infty(Q)$ naturally arises. In order to solve this problem it is necessary to construct corresponding theory of KHM over the algebras $L^0(\Omega, \Sigma, \mu)$ and $C_\infty(Q)$. In a particular case of KHM over the algebras $L^0(\Omega, \Sigma, \mu)$ this problem is solved in \cite{lit3}, where the decomposition of KHM over $L^0(\Omega, \Sigma, \mu)$ as a direct sum of homogeneous KHM is given. Similar decomposition as a direct sum of strictly $\gamma$-homogeneous modules is given in the paper \cite{lit7a} for arbitrary regular laterally complete modules over the algebra $C_\infty(Q)$ (the definitions see in the Section 3 below).

The algebra $C_\infty(Q)$ is an example of a commutative unital regular algebra over the field of real numbers. In this algebra the following property of lateral completeness holds: for any set $\{a_i\}_{i \in I}$ of pairwise disjoint elements in $C_\infty(Q)$ there exists an element $a \in C_\infty(Q)$ such that $as(a_i) = a_i$ for all $i \in I$, where $s(a_i)$ is a support of the element $a_i$ (the definitions see in the Section 2 below). This property of $C_\infty(Q)$ plays a crucial role in classification of regular laterally complete $C_\infty(Q)$-mpdules \cite{lit7a}. Thereby, it is natural to consider the class of laterally complete commutative unital regular algebras $\mathcal{A}$ over arbitrary fields and to obtain variants of structure theorems for modules over such algebras. Current work is devoted to solving this problem. For every faithful regular laterally complete $\mathcal{A}$-module $X$ the concept of passport $\Gamma(X)$, which consist of the uniquely defined partition of unity in the Boolean algebra of idempotents in $\mathcal{A}$ and the set of pairwise different cardinal numbers is constructed. It is proved, that the equality of passports $\Gamma(X)$ and $\Gamma(Y)$ is necessary and sufficient condition for isomorphism of $\mathcal{A}$-modules $X$ and $Y$.

\section{Laterally complete commutative regular algebras}
\label{sec:1}
Let $\mathcal{A}$ be a commutative algebra over the field $K$ with the unity $\mathbf{1}$ and $\nabla=\{e\in\mathcal{A}: e^2=e\}$ be a set of all idempotents in $\mathcal{A}$. For all $e, f \in\nabla$ we write $e \leq f$ if $ef=e$. It is well known (see, for example \cite[Prop.~1.6]{lit11}) that this binary relation is partial order in $\nabla$ and $\nabla$ is a Boolean algebra with respect to this order. Moreover, we have the following equalities: $e \vee f = e + f - ef$, $e \wedge f = e f$, $Ce = \mathbf{1} - e$ with respect to the lattice operations and the complement $Ce$ in $\nabla$.

The commutative unital algebra $\mathcal{A}$ is called regular if the following equivalent conditions hold \cite[\S2, item 4]{lit12}:

1. For any $a\in\mathcal{A}$ there exists $b\in\mathcal{A}$ such that $a=a^2b$;

2. For any $a\in\mathcal{A}$ there exists $e\in\nabla$ such that $a\mathcal{A}=e\mathcal{A}$.

A regular algebra $\mathcal{A}$ is a regular semigroup with respect to the multiplication operation \cite[Ch. I, \S 1.9]{lit13}. In this case all idempotents in $\mathcal{A}$ commute pairwisely. Therefore, $\mathcal{A}$ is a commutative inverse semigroup, i.e. for any $a\in\mathcal{A}$ there exists an unique element $i(a)\in\mathcal{A}$, which is an unique solution of the system: $a^2x=a$, $ax^2=x$ \cite[Ch. I, \S 1.9]{lit13}. The element $i(a)$ is called an inversion of the element $a$. Obviously, $ai(a)\in\nabla$ for any $a\in\mathcal{A}$. In this case the map $i:\mathcal{A}\rightarrow\mathcal{A}$ is a bijection and an automorphism (by multiplication) in semigroup $\mathcal{A}$. Moreover, $i(i(a))=a$ and $i(g)=g$ for all $a\in\mathcal{A}$, $g\in\nabla$.

Let $\mathcal{A}$ be a commutative unital regular algebra and $\nabla$ be a Boolean algebra of all idempotents in $\mathcal{A}$. Idempotent $s(a)\in\nabla$ is called the support of an element $a\in\mathcal{A}$ if $s(a)a=a$ and $ga=a$, $g\in\nabla$ imply $s(a) \leq g$. It is clear that $s(a)=ai(a)=s(i(a))$. In particular, $s(e)=ei(e)=e$ for any $e\in\nabla$.

It is easy to show that supports of elements in a commutative regular unital algebra $\mathcal{A}$ satisfy the following properties:

\begin{proposition}\label{art9_utv_2_2}
Let $a,b\in \mathcal{A}$, then

(i). $s(ab) = s(a)s(b)$, in particular, $ab = 0 \Leftrightarrow s(a)s(b) = 0$;

(ii). If $ab = 0$, then $i(a+b) = i(a) + i(b)$ and $s(a+b) = s(a) + s(b)$.
\end{proposition}

Two elements $a$ and $b$ in a commutative unital regular algebra $\mathcal{A}$ are called \emph{disjoint elements}, if $ab = 0$, which equivalent to the equality $s(a)s(b) = 0$ (see Proposition \ref{art9_utv_2_2} (i)). If the Boolean algebra $\nabla$ of all idempotents in $\mathcal{A}$ is complete, $a\in\mathcal{A}$ and $r(a)=\sup\{e\in\nabla: ae = 0\}$, then
\begin{multline*}
s(a)r(a) = s(a) \wedge r(a) = s(a) \wedge (\sup\{e: ae = 0\})=\\
= \sup\{s(a) \wedge e: ae = 0\} = \sup\{s(a)e: ae = 0\} = 0.
\end{multline*}
Hence $s(a) \leq \mathbf{1} - r(a)$. If $q = (\mathbf{1} - r(a) - s(a))$, then $aq = as(a)q = 0$, thus $q \leq r(a)$. This yields that $q = 0$, i.e. $s(a) = \mathbf{1} - r(a)$. This implies the following

\begin{proposition}\label{art9_utv_2_25}
Let $\mathcal{A}$ be a commutative unital regular algebra and let $\nabla$ be complete Boolean algebra of idempotents in $\mathcal{A}$. If $\{e_i\}_{i \in I}$ is a partition of unity in $\nabla$, $a,b \in \mathcal{A}$ and $ae_i=be_i$ for all $i \in I$, then $a=b$.
\end{proposition}

\begin{proof}
Since $(a-b)e_i=0$ for any $i \in I$, then $\mathbf{1}=\sup\limits_{i \in I}e_i \leq r(a-b)$, i.e. $r(a-b)=\mathbf{1}$. Hence, $s(a-b)=0$, i.e. $a=b$.
\end{proof}

Commutative unital regular algebra $\mathcal{A}$ is called \emph{laterally complete} (\emph{$l$-complete}) if the Boolean algebra of its idempotents is complete and for any set $\{a_i\}_{i \in I}$ of pairwise disjoint elements in $\mathcal{A}$ there exists an element $a\in\mathcal{A}$ such that $as(a_i) = a_i$ for all $i \in I$. The element $a\in\mathcal{A}$ such that $as(a_i) = a_i$, $i \in I$, in general, is not uniquely determined. However, by Proposition \ref{art9_utv_2_25}, it follows that the element $a$ is unique in the case, when $\sup\limits_{i \in I}s(a_i) = \mathbf{1}$. In general case, due to the equality $as(a_i) = a_i = bs(a_i)$ for all $i \in I$ and $a,b\in\mathcal{A}$, it follows that $a\sup_{i \in I}s(a_i) = b\sup_{i \in I}s(a_i)$.

Let us give examples of $l$-complete and not $l$-complete commutative regular algebras. Let $\Delta$ be an arbitrary set and $K^\Delta$ be a Cartesian product of $\Delta$ copies of the field $K$, i.e. the set of all $K$-valued functions on $\Delta$. The set $K^\Delta$ is a commutative unital regular algebra with respect to pointwise algebraic operations, moreover, the Boolean algebra $\nabla$ of all idempotents in $K^\Delta$ is an isomorphic atomic Boolean algebra of all subsets in $\Delta$. In particular $\nabla$ is complete Boolean algebra. If $\{a_j=(\alpha^{(j)}_q)_{q\in\Delta}, j \in J\}$ is a family of pairwise disjoint elements in $K^\Delta$, then setting $\Delta_j = \{q\in\Delta: \alpha^{(j)}_q \neq 0\}$, $j \in J$ and $a=(\alpha_q)_{q\in\Delta} \in K^\Delta$, where $\alpha_q = \alpha^{(j)}_q$ for any $q\in\Delta_j$, $j \in J$, and $\alpha_q = 0$ for $q\in\Delta\setminus\bigcup\limits_{j \in J}\Delta_j$, we obtain that $as(a_j) = a_j$ for all $j \in J$. Hence, $K^\Delta$ is a $l$-complete algebra.

Now let $\mathcal{A}$ be an arbitrary commutative unital regular algebra over the field $K$ and $\nabla$ be a Boolean algebra of all idempotents in $\mathcal{A}$. An element $a\in\mathcal{A}$ is called a \emph{step element} in $\mathcal{A}$ if it has the following form $a = \sum_{k=1}^n\lambda_k e_k$, here $\lambda_k \in K$, $e_k\in\nabla$, $k=1,\dots,n$. The set $K(\nabla)$ of all step elements is the smallest subalgebra in $\mathcal{A}$, which contains $\nabla$. Any nonzero element $a = \sum_{k=1}^n \lambda_k e_k$ in $K(\nabla)$ can be represented as $a = \sum_{l=1}^m\alpha_l g_l$, here $g_l\in\nabla$, $g_l g_k = 0$ when $l \ne k$, $0 \ne \alpha_k \in K$, $l,k=1,\dots,m$. Setting $b = \sum_{l=1}^m\alpha_l^{-1} g_l \in K(\nabla)$, we obtain $a^2b=a$. Hence, $K(\nabla)$ is a regular subalgebra in $\mathcal{A}$. Since $\nabla \subset K(\nabla)$, the Boolean algebra of idempotents in $K(\nabla)$ coincides with $\nabla$. Assume that $\mathrm{card}\,(K) = \infty$ and $\mathrm{card}\,(\nabla) = \infty$. We choose a countable set $K_0 = \{\lambda_n\}_{n=1}^\infty$ of pairwise different nonzero elements in $K$ and a countable set $\{e_n\}_{n=1}^\infty$ of nonzero pairwise disjoint elements in $\nabla$. Let us consider a set $\{\lambda_n e_n\}_{n=1}^\infty$ of pairwise disjoint elements in $K(\nabla)$. Assume that there exists $b=\sum_{l=1}^m \alpha_l g_l\in K(\nabla)$, $0\ne\alpha_l \in K$, $g_l\in\nabla$, $g_l g_k = 0$ and $l \ne k$, $l,k=1,\dots,m$, such that $be_n=bs(\lambda_n e_n)=\lambda_n e_n$. In this case for any positive integer $n$ there exists natural number $l(n)$, such that $\alpha_{l(n)}g_{l(n)}e_n=\lambda_ng_{l(n)}e_n\neq 0$, i.e. $\alpha_{l(n)}=\lambda_n$. This implies that the set $\{\lambda_n\}_{n=1}^\infty$ is finite, which is not true. Hence, the commutative unital regular algebra $K(\nabla)$ is not $l$-complete.

Let $\nabla$ be complete Boolean algebra and let $Q(\nabla)$ be a Stone compact corresponding to $\nabla$. An algebra $C_\infty(Q(\nabla))$ of all continuous functions $a: Q(\nabla) \rightarrow [-\infty,+\infty]$, taking the values $\pm\infty$ only on nowhere dense sets in $Q(\nabla)$ \cite[1.4.2]{lit5}, is an important example of a $l$-complete commutative regular algebra.

An element $e\in C_\infty(Q(\nabla))$ is an idempotent if and only if $e(t) = \chi_V(t)$, $t\in Q(\nabla)$, for some clopen set $V \subset Q(\nabla)$, where
\[ \chi_V(t) = \left\{
\begin{array}{ll}
1, & t \in V\textrm{;}\\
0, & t \notin V\textrm{,}
\end{array} \right. \]
i.e. $\chi_V(t)$ is a characteristic function of the set $V$. In particular, the Boolean algebra $\nabla$ can be identified with the Boolean algebra of all idempotents in algebra $C_\infty(Q(\nabla))$.

If $a\in C_\infty(Q(\nabla))$, then $G(a) = \{t \in Q(\nabla): 0<|a(t)|<+\infty\}$ is open set in the Stone compact set $Q(\nabla)$. Hence, the closure $V(a) = \overline{G(a)}$ in $Q(\nabla)$ of the set $G(a)$ is an clopen set, i.e. $\chi_{V(a)}$ is an idempotent in the algebra $C_\infty(Q(\nabla))$. We consider a continuous function $b(t)$, given on the dense open set $G(a) \cup (Q(\nabla) \setminus V(a))$ and defines by the following equation
\[ b(t) = \left\{
\begin{array}{ll}
\frac{1}{a(t)}, & t \in G(a)\textrm{,}\\
0, & t \in Q(\nabla) \setminus V(a)\textrm{.}
\end{array} \right. \]
This function uniquely extends to a continuous function defined on $Q(\nabla)$ with values in $[-\infty,+\infty]$ \cite[Ch.5, \S 2]{lit14} (we also denote this extension by $b(t)$). Since $ab=\chi_{V(a)}$, then $a^2b = a$ and $s(a) = \chi_{V(a)}$. Hence, $C_\infty(Q(\nabla))$ is a commutative unital regular algebra over the field of real numbers $\mathbf{R}$. In this case, the Boolean algebra of all idempotents in $C_\infty(Q(\nabla))$ is complete.

It is known that (see, for example \cite[1.4.2]{lit5}) $C_\infty(Q(\nabla))$ is an extended complete vector lattice. In particular, for any set $\{a_j\}_{j\in J}$ of pairwise disjoint positive elements in $C_\infty(Q(\nabla))$ there exists the least upper bound $a=\sup_{j\in J}a_j$ and $as(a_j) = a_j$ for all $j \in J$. It follows that the commutative regular algebra $C_\infty(Q(\nabla))$ is laterally complete.

In the case, when $\nabla$ is a complete atomic Boolean algebra and $\Delta$ is the set of all atoms in $\nabla$, then $C_\infty(Q(\nabla))$ is isomorphic to the algebra $\mathbf{R}^\Delta$.

The following examples of laterally complete commutative regular algebras are variants of algebras $C_\infty(Q(\nabla))$ for any topological fields, in particular, for the field $\mathbf{Q}_p$ of $p$-adic numbers.

Let $K$ be an arbitrary field and $t$ be the Hausdorff topology on $K$. If operations $\alpha\rightarrow(-\alpha)$, $\alpha\rightarrow\alpha^{-1}$ and operations $(\alpha,\beta)\rightarrow\alpha+\beta$, $(\alpha,\beta)\rightarrow\alpha\beta$, $\alpha,\beta \in K$, are continuous with respect to this topology, we say that $(K,t)$ is a \emph{topological field} (see, for example, \cite[Ch.20, \S165]{lit15}).

Let $(K,t)$ be a topological field, $(X,\tau)$ be any topological space and $\nabla(X)$ be a Boolean algebra of all clopen subsets in $(X,\tau)$. A map $\varphi:(X,\tau)\rightarrow(K,t)$ is called \emph{almost continuous} if there exists a dense open set $U$ in $(X,\tau)$ such that the restriction $\varphi|_U: U \rightarrow (K,t)$ of the map $\varphi$ on the subset $U$ is continuous in $U$. The set of all almost continuous maps from $(X,\tau)$ to $(K,t)$ we denote by $AC(X,K)$.

We define pointwise algebraic operations in $AC(X,K)$ by
\[ (\varphi+\psi)(t) = \varphi(t)+\psi(t); \]
\[ (\alpha\varphi)(t) = \alpha\varphi(t); \]
\[ (\varphi\cdot\psi)(t) = \varphi(t)\psi(t) \]
for all $\varphi,\psi \in AC(X,K)$, $\alpha\in K$, $t\in X$.

Since an intersection of two dense open sets in $(X,\tau)$ is a dense open set in $(X,\tau)$, then $\varphi+\psi$, $\varphi\cdot\psi \in AC(X,K)$ for any $\varphi, \psi \in AC(X,K)$. Obviously, $\alpha\varphi \in AC(X,K)$ for all $\varphi \in AC(X,K)$, $\alpha \in K$. It can be easily checked that $AC(X,K)$ is a commutative algebra over $K$ with the unit element $\mathbf{1}(t) = 1_K$ for all $t \in X$, where $1_K$ is the unit element of $K$. In this case, the algebra $C(X,K)$ of all continuous maps from $(X,\tau)$ to $(K,t)$ is a subalgebra in $AC(X,K)$.

\begin{sloppypar}
In the algebra $AC(X,K)$ consider the following ideal
\[ I_0(X,K) = \{\varphi \in AC(X,K):\ \text{interior of preimage}\ \varphi^{-1}(0)\ \text{is dense in}\ (X,\tau)\}. \]
By $C_\infty(X,K)$ denote the quotient algebra $AC(X,K)/I_0(X,K)$ and by 
$$
\pi: AC(X,K) \rightarrow AC(X,K)/I_0(X,K)
$$ 
denote the corresponding canonical homomorphism.
\end{sloppypar}

\begin{theorem}\label{art9_teor_2_26}
The quotient algebra $C_\infty(X,K)$ is a commutative unital regular algebra over the field $K$. Moreover, if $(X,\tau)$ is a Stone compact set, then algebra $C_\infty(X,K)$ is laterally complete, and the Boolean algebra $\nabla$ of all its idempotents is isomorphic to the Boolean algebra $\nabla(X)$.
\end{theorem}

\begin{proof}
Since $AC(X,K)$ is a commutative unital algebra over $K$, then $C_\infty(X,K)$ is also a commutative unital algebra over $K$ with unit element $\pi(\mathbf{1})$. Now we show that $C_\infty(X,K)$ is a regular algebra, i.e. for any $\varphi \in AC(X,K)$ there exists $\psi \in AC(X,K)$, such that $\pi^2(\varphi)\pi(\psi)=\pi(\varphi)$.

We fix an element $\varphi \in AC(X,K)$ and choose a dense open set $U\in\tau$, such that the restriction $\varphi|_U: U \rightarrow (K,t)$ is continuous. Since $K\setminus\{0\}$ is an open set in $(K,t)$, then the set $V = U\cap\varphi^{-1}(K\setminus\{0\})$ is open in $(X,\tau)$. Clearly, the set $W = X \setminus \overline{V}^\tau$ is also open in $(X,\tau)$, in this case $V \cup W$ is a dense open set in $(X,\tau)$.

We define a map $\psi: X \rightarrow K$, as follow: $\psi(x) = (\varphi(x))^{-1}$ if $x \in V$, and $\psi(x) = 0$ if $x \in X \setminus V$. It is clear  that $\psi \in AC(X,K)$ and $\varphi^2\psi-\varphi \in I_0(X,K)$, i.e. $\pi^2(\varphi)\pi(\psi)=\pi(\varphi)$. Hence, the algebra $C_\infty(X,K)$ is regular.

For any clopen set $U\in\nabla(X)$ its characteristic function $\chi_U$ belongs to $AC(X,K)$, in this case, $\pi(\chi_U)^2 = \pi(\chi_U^2) = \pi(\chi_U)$, i.e. $\pi(\chi_U)$ is an idempotent in the algebra $C_\infty(X,K)$.

Assume that $(X,\tau)$ is a Stone compact and we show that for any idempotent $e \in C_\infty(X,K)$ there exists $U \in \nabla(X)$ such that $e = \pi(\chi_U)$.

If $e\in\nabla$, then $e = \pi(\varphi)$ for some $\varphi \in AC(X,K)$ and 
$$\pi(\varphi) = e^2 = \pi(\varphi^2),
$$
i.e. $(\varphi^2-\varphi) \in I_0(X,K)$. Hence, there exists a dense open set $V$ in $X$ such that $\varphi^2(t)-\varphi(t) = 0$ for all $t \in V$. Denote by $U$ a dense open set in $X$ such that the restriction $\varphi|_U: U \rightarrow K$ is continuous. Put $U_0 = \varphi^{-1}(\{0\})\cap(U \cap V)$, $U_1=\varphi^{-1}(\{1_K\})\cap(U \cap V)$. Since $U_0 \cap U_1 = \emptyset, U_0 \cup U_1 = U \cap V \in \tau$ and the sets $U_0$, $U_1$ are closed in $U \cap V$ with respect to the topology induced from $(X, \tau)$, it follows that $U_0, U_1 \in \tau$. Hence, the set $U_\varphi = \overline{U_1}$ belongs to the Boolean algebra $\nabla(X)$, besides, $U_\varphi \cap U_0 = \emptyset$.

Since $U_0\cup U_1=U\cap V$ is a dense open set in $(X, \tau)$ and $\varphi(t)=\chi_{U_\varphi}(t)$ for all $t \in U_0\cup U_1$, it follows that $e=\pi(\varphi)=\pi(\chi_{U_\varphi})$. Thus, the mapping $\Phi:\nabla(X)\rightarrow\nabla$ defined by the equality $\Phi(U)=\pi(\chi_U)$, $U \in\nabla(X)$, is a surjection.

Moreover, for $U,V\in\nabla(X)$ the following equalities hold
\[ \Phi(U \cap V) = \pi(\chi_{U \cap V}) = \pi(\chi_U\chi_V) = \pi(\chi_U)\pi(\chi_V) = \Phi(U)\Phi(V), \]
\[ \Phi(X \setminus U) = \pi(\chi_{X \setminus U}) = \pi(\mathbf{1} - \chi_U) = \Phi(X)-\Phi(U). \]
Furthermore, the equality $\Phi(U)=\Phi(V)$ implies that the continuous mappings $\chi_U$ and $\chi_V$ coincide on a dense set in $X$. Therefore $\chi_U=\chi_V$, that is $U=V$.

Hence, $\Phi$ is an isomorphism from the Boolean algebra $\nabla(X)$ onto the Boolean algebra $\nabla$ of all idempotents from $C_\infty(X,K)$, in particular, $\nabla$ is a complete Boolean algebra.

Finally, to prove $l$-completeness of the algebra $C_\infty(X,K)$ we show that for any family $\{\pi(\varphi_i): \varphi\in AC(X,K)\}_{i\in I}$ of nonzero pairwise disjoint elements in $C_\infty(X,K)$ there exists $\varphi \in AC(X,K)$ such that $\pi(\varphi)s(\pi(\varphi_i)) = \pi(\varphi_i)$ for all $i \in I$. For any $i \in I$ we choose a dense open set $U_i$ such that the restriction $\varphi_i|_{U_i}$ is continuous and put $V_i = U_i\cap\varphi_i^{-1}(K\setminus\{0\})$, $i \in I$. It is not hard to prove that $s(\pi(\varphi_i)) = \Phi(\overline{V_i})$. In particular, $V_i \cap V_j = \emptyset$ when $i \ne j$, $i,j \in I$. Define the mapping $\varphi: X \rightarrow K$, as follows $\varphi(t) = \varphi_i(t)$ if $t \in V_i$ and $\varphi(t) = 0$ if $t \in X\setminus\left(\bigcup\limits_{i \in I}V_i\right)$. Clearly, $\varphi \in AC(X,K)$ and $\pi(\varphi)s(\pi(\varphi_i)) = \pi(\varphi\chi_{\overline{V_i}}) = \pi(\varphi_i\chi_{\overline{V_i}}) = \pi(\varphi_i)$ for all $i \in I$.

\end{proof}

\section{Laterally complete regular modules}
\label{sec:2}
Let $\mathcal{A}$ be a laterally complete commutative regular algebra and let $\nabla$ be a Boolean algebra of all idempotents in $\mathcal{A}$. Let $X$ be a left $\mathcal{A}$-module with algebraic operations $x + y$ and $a x$, $x,y \in X$, $a\in\mathcal{A}$. Since the algebra $\mathcal{A}$ is commutative, then a left $\mathcal{A}$-module $X$ becomes a right $\mathcal{A}$-module, if we put $x a := a x$, $x \in X$, $a \in \mathcal{A}$. Hence, we can assume, that $X$ is a bimodule over $\mathcal{A}$, where the following equality $a x = x a$ holds for any $x \in X$, $a\in\mathcal{A}$. Next, an $\mathcal{A}$-bimodule $X$ we shall call an $\mathcal{A}$-module.

An $\mathcal{A}$-module $X$ is called faithful, if for any nonzero $e \in \nabla$ there exists $x \in X$ such that $ex \ne 0$. Clearly, for a faithful $\mathcal{A}$-module $X$ the set $X_e := eX$ is a faithful $\mathcal{A}_e$-module for any $0 \ne e \in \nabla$, where $\mathcal{A}_e := e\mathcal{A}$.

An $\mathcal{A}$-module $X$ is said to be a regular module, if for any $x\in\mathcal{A}$ the condition $ex = 0$ \ for all $e \in L \subset \nabla$ implies $(\sup L)x = 0$. In this case, for $x \in X$ the idempotent
$$
s(x) = \mathbf{1} - \sup\{e \in \nabla: ex = 0\}
$$
is called the support of an element $x$. In case, when $X = \mathcal{A}$, the notions of support of an element in an $\mathcal{A}$-module $X$ and of support of an element in $\mathcal{A}$ coincide. If $X$ is a regular $\mathcal{A}$-module, then $X_e$ is also a regular $\mathcal{A}_e$-module for any nonzero $e\in\nabla$.

We need the following properties of supports of elements in a regular $\mathcal{A}$-module $X$.

\begin{proposition}\label{art9_utv_2_3}
Let $X$ be a regular $\mathcal{A}$-module, $x,y \in X$, $a \in \mathcal{A}$. Then

(i). $s(x)x = x$;

(ii). if $e\in\nabla$ and $ex = x$, then $e \geq s(x)$;

(iii). $s(a x) = s(a)s(x)$.
\end{proposition}

\begin{proof}
$(i)$. If $r(x) = \sup\{e\in\nabla: ex = 0\}$, then $s(x) = \mathbf{1} - r(x)$ and $r(x)x = 0$. Hence, $x = (s(x) + r(x))x = s(x)x$.

$(ii)$. As $ex = x$, then $(\mathbf{1} - e)x = 0$, and therefore $\mathbf{1} - e \leq r(x)$. Thus $e \geq \mathbf{1} - r(x) = s(x)$.

$(iii)$. Since $(s(a)s(x)) \cdot (ax) = (s(a)a) \cdot (s(x)x) = ax$, then by $(ii)$ we have $s(a x) \leq s(a)s(x)$. If $g = s(a)s(x) - s(a x) \ne 0$, then $ga \ne 0$, $g \leq s(a)$ and $gs(ax)=0$. Hence $gax = 0$ and $0 = i(ga)(gax) = (i(g)i(a)ga)x = (gi(a)a)x = gs(a)x = gx \ne 0$. This contradiction implies $g = 0$, i.e. $s(a x) = s(a)s(x)$.
\end{proof}

We say that a regular \ $\mathcal{A}$-module \ $X$ \ is laterally complete \ ($l$-complete), if for any set \  $\{x_i\}_{i \in I} \subset X$ \ and for any partition $\{e_i\}_{i \in I}$ of unity of the Boolean algebra $\nabla$ there exists $x \in X$ such that $e_i x = e_i x_i$ for all $i \in I$. In this case, the element $x$ is called mixing of the set $\{x_i\}_{i \in I}$ with respect to the partition of unity $\{e_i\}_{i \in I}$ and denote by $\underset{i \in I}{\mathrm{mix}}(e_i x_i)$. Mixing $\underset{i \in I}{\mathrm{mix}}(e_i x_i)$ is defined uniquely, whereas the equalities $e_i x = e_i x_i = e_i y$, $x, y \in X$, $i \in I$, implies $e_i (x - y) = 0$ for all $i \in I$, and, by regularity of the $\mathcal{A}$-module $X$, we obtain $x = y$.

Let $\{x_i\}_{i \in I} \subset E \subset X$ and let $\{e_i\}_{i \in I}$ be a partition of unity in $\nabla$. The set of all mixings $\underset{i \in I}{\mathrm{mix}}(e_i x_i)$ is called a cyclic hull of the set $E$ in $X$ and denotes by $\mathrm{mix}(E)$. Obviously, the inclusion $E \subset \mathrm{mix}(E)$ is always true. If $E = \mathrm{mix}(E)$, then $E$ is called a cyclic set in $X$ (compare with \cite{lit4}, 1.1.2).

Thus, a regular $\mathcal{A}$-module $X$ is a $l$-complete $\mathcal{A}$-module if and only if $X$ is a cyclic set. In particular, in any $l$-complete $\mathcal{A}$-module $X$ its submodule $X_e$ is also a $l$-complete $\mathcal{A}_e$-module for any nonzero idempotent $e$ in $\mathcal{A}$.

We need the following properties of cyclic hulls of sets.

\begin{proposition}\label{art9_utv_2_6}
Let $X$ be a $l$-complete $\mathcal{A}$-module and let $E$ be a nonempty subset in $X$, $a \in \mathcal{A}$. Then

(i). $\mathrm{mix}(\mathrm{mix}(E)) = \mathrm{mix}(E)$;

(ii). $\mathrm{mix}(a E) = a \mathrm{mix}(E)$;

(iii). If $Y$ is an $\mathcal{A}$-submodule in $X$, then $\mathrm{mix}(Y)$ is a $l$-complete $\mathcal{A}$-submodule in $X$;

(iv). If $U$ is an isomorphism from $\mathcal{A}$-module $X$ onto $\mathcal{A}$-module $Z$, then $Z$ is a $l$-complete $\mathcal{A}$-module and $\mathrm{mix}(U(E)) = U(\mathrm{mix}(E))$.
\end{proposition}

\begin{proof}
$(i)$. It is sufficient to show that $\mathrm{mix}(\mathrm{mix}(E)) \subset \mathrm{mix}(E)$. If $x\in\mathrm{mix}(\mathrm{mix}(E))$, then $x = \underset{i \in I}{\mathrm{mix}}(e_i x_i)$, where $x_i \in \mathrm{mix}(E)$, $i \in I$. Since $x_i\in\mathrm{mix}(E)$, then $x_i = \underset{j \in J(i)}{\mathrm{mix}}(e_j^{(i)} x_j^{(i)})$, where $x_j^{(i)} \in E$, $j \in J(i)$ and $\{e_j^{(i)}\}_{j \in J(i)}$ is a partition of unity in the Boolean algebra $\nabla$ for all $i \in I$. Fix $i \in I$ and put $g_j^{(i)} := e_i e_j^{(i)}$. It is clear that $\{g_j^{(i)}\}_{j \in J(i)}$ is a partition of the idempotent $e_i$. Hence, $\{g_j^{(i)}\}_{j \in J(i), i \in I}$ is a partition of unity $\mathbf{1}$. Besides, 
$$
g_j^{(i)}x = g_j^{(i)} e_i x = g_j^{(i)} e_i x_i = e_i e_j^{(i)} x_i = e_i e_j^{(i)} x_j^{(i)} = g_j^{(i)} x_j^{(i)}.
$$
This yields that $x = \underset{j \in J(i), i \in I}{\mathrm{mix}}(g_j^{(i)}x_j^{(i)}) \in \mathrm{mix}(E)$.

$(ii)$. If $x\in\mathrm{mix}(a E)$, then $x = \underset{i \in I}{\mathrm{mix}}(e_i ay_i)$, where $y_i \in E, i \in I$. Since $X$ is a $l$-complete $\mathcal{A}$-module, then there exists $y = \underset{i \in I}{\mathrm{mix}}(e_i y_i) \in \mathrm{mix}(E)$ and $e_i x = a e_i y_i = e_i (a y)$ for all $i \in I$. Hence, $e_i(x - a y) = 0$, and regularity of the $\mathcal{A}$-module $X$ implies the equality $x = a y$. Thus, $\mathrm{mix}(a E)\subset a\mathrm{mix}(E)$.

Conversely, if $x\in a \mathrm{mix}(E)$, then $x = a z$, where $z = \underset{i \in I}{\mathrm{mix}}(e_i z_i), z_i \in E, i \in I$. Since $a z_i \in a E$ and $e_i x = e_i (a z) = e_i a e_i z = e_i (a z_i)$ for all $i \in I$, we have that $x = \underset{i \in I}{\mathrm{mix}}(e_i(a z_i)) \in \mathrm{mix}(a E)$. Hence, $a \mathrm{mix}(E) \subset \mathrm{mix}(a E)$.

$(iii)$. Let $x, y \in \mathrm{mix}(Y)$, $x = \underset{i \in I}{\mathrm{mix}}(e_i x_i)$, $y = \underset{j \in J}{\mathrm{mix}}(g_j y_j)$, where $x_i, y_j \in Y$, $i \in I$, $j \in J$, $\{e_i\}_{i \in I}$, $\{g_j\}_{j \in J}$ are partitions of unity in $\nabla$. Clearly, that $p_{ij} = e_i g_j$, $i \in I$, $j \in J$, is also a partition of unity in $\nabla$ and $p_{ij}(x+y)=p_{ij}(x_i+y_j)$, where $x_i + y_j \in Y$ for all $i \in I$, $j \in J$. This means that $(x + y)\in\mathrm{mix}(Y)$.

Since $a Y \subset Y$, then by $(ii)$ we have that $a x \in a \mathrm{mix}(Y) = \mathrm{mix}(a Y) \subset \mathrm{mix}(Y)$. Hence, $\mathrm{mix}(Y)$ is an $\mathcal{A}$-submodule in $X$, and by regularity of the $\mathcal{A}$-module $X$, it is a regular $\mathcal{A}$-module. The equality $\mathrm{mix}(Y) = \mathrm{mix}(\mathrm{mix}Y)$ (see $(i)$) implies that $\mathrm{mix}Y$ is a $l$-complete $\mathcal{A}$-module.

$(iv)$. If $U(x) = y \in Z$, $x \in X$, $\emptyset \ne L \subset \nabla$ and $ey = 0$ for all $e \in L$, then $U(ex) = e U(x) = ey = 0$. Since $U$ is a bijection, then $ex = 0$ for any $e \in L$. By regularity of the $\mathcal{A}$-module $X$, we have that $(\sup L)x = 0$, and, therefore, $(\sup L)y = U((\sup L)x) = 0$. Hence, $Z$ is a regular $\mathcal{A}$-module. In the same way we show that $Z$ is a $l$-complete $\mathcal{A}$-module and the equality $\mathrm{mix}(U(E)) = U(\mathrm{mix}(E))$ holds.

\end{proof}

Let $\nabla$ be an arbitrary complete Boolean algebra. For any nonzero element $e\in\nabla$ we put $\nabla_e = \{q\in\nabla: q \leq e\}$. The set $\nabla_e$ is a Boolean algebra with the unity $e$ with respect to partial order, induced from $\nabla$.

We say that a set $B$ in $\nabla$ is a minorant subset for nonempty set $E\subset\nabla$, if for any nonzero $e \in E$ there exists nonzero $q \in B$ such that $q \leq e$. We need the following property of complete Boolean algebras.

\begin{theorem} {\rm (\cite{lit5}, 1.1.6)} \label{art9_teor_2_1}
If $\nabla$ is a complete Boolean algebra, $e$ is a nonzero element in $\nabla$ and $B$ is a minorant subset for $\nabla_e$, then there exists a disjoint subset $L \subset B$ such that $\sup L = e$.
\end{theorem}

We say that a Boolean algebra $\nabla$ has \emph{a countable type} or is \emph{$\sigma$-finite}, if any nonfinite family of nonzero pairwise disjoint elements in $\nabla$ is a countable set. A complete Boolean algebra $\nabla$ is called \emph{multi-$\sigma$-finite}, if for any nonzero element $g\in\nabla$ there exists $0 \ne e\in\nabla$ such that $e \leq g$ and the Boolean algebra $\nabla_e$ has a countable type. By theorem \ref{art9_teor_2_1}, a multi-$\sigma$-finite Boolean algebra $\nabla$ always has a partition $\{e_i\}_{i \in I}$ of unity $\mathbf{1}$ such that the Boolean algebra $\nabla_{e_i}$ has a countable type for all $i \in I$.

By theorem \ref{art9_teor_2_1} we set the following useful properties of $l$-complete $\mathcal{A}$-modules.

\begin{proposition}\label{art9_utv_2_7}
Let $X$ be an arbitrary $l$-complete $\mathcal{A}$-module and $\nabla$ be a complete Boolean algebra of all idempotents in $\mathcal{A}$. Then

(i). If $X$ is a faithful $\mathcal{A}$-module, then there exists an element $x \in X$ such that $s(x) = \mathbf{1}$;

(ii). If $Y$ is a $l$-complete $\mathcal{A}$-submodule in a regular $\mathcal{A}$-module $X$ and for any nonzero $e \in \nabla$ there exists a nonzero $g_e \in \nabla$ such that $g_e \leq e$ and $g_e Y = g_e X$, then $Y = X$.
\end{proposition}

\emph{Proof} is in the same way as the proof of Proposition 2.4 in \cite{lit7a}.

We need a representation of a faithful $l$-complete $\mathcal{A}$-module $X$ as the Cartesian product of a faithful $l$-complete $\mathcal{A}_{e_i}$-modules family, where $\{e_i\}_{i \in I}$ is a partition of unity in the Boolean algebra $\nabla$ of all idempotents in $\mathcal{A}$. In the Cartesian product
\[ \prod_{i \in I}e_i X = \{\{y_i\}_{i \in I}: y_i \in e_i X\} \]
of $\mathcal{A}$-submodules $e_i X$ we consider coordinate-wise algebraic operations. It is clear that $\prod\limits_{i \in I}e_i X$ is a faithful $l$-complete $\mathcal{A}$-module. We define a map $U: X \rightarrow \prod\limits_{i \in I}e_i X$ given by $U(x) = \{e_i x\}_{i \in I}$. Obviously, $U$ is a homomorphizm from $X$ onto $\prod\limits_{i \in I}e_iX$. If $U(x) = U(y)$, then $e_i x = e_i y$ for all $i \in I$, and by regularity of the $\mathcal{A}$-module $X$, it follows that $x = y$.

If $z = \{x_i\}_{i \in I} \in \prod\limits_{i \in I}e_i X$, where $x_i \in e_i X \subset X$, $i \in I$, then $l$-completeness of the $\mathcal{A}$-module $X$ implies that there exists an element $x \in X$ such that $e_i x = e_i x_i = x_i$ for all $i \in I$. Hence, $U(x) = z$, i.e. $U$ is a surjection.

Thus, the following proposition holds.

\begin{proposition}\label{art9_utv_2_8}
If $X$ is a faithful $l$-complete $\mathcal{A}$-module, $\{e_i\}_{i \in I}$ is a partition of unity of the Boolean algebra $\nabla$ of all idempotents in $\mathcal{A}$, then $\prod\limits_{i \in I}e_i X$ is also a faithful $l$-complete $\mathcal{A}$-module and $U$ is an isomorphism from $X$ onto $\prod\limits_{i \in I}e_i X$.
\end{proposition}

\section{Homogenous $\mathcal{A}$-modules}
\label{sec:3}
Let $\mathcal{A}$ be a laterally complete commutative regular algebra, let $\nabla$ be a complete Boolean algebra of all idempotents in $\mathcal{A}$, let $X$ be a faithful $\mathcal{A}$-module. The following $\mathcal{A}$-submodule in $X$ is called $\mathcal{A}$-linear hull of a nonempty subset $Y \subset \mathcal{A}$
\[ \mathrm{Lin}(Y, \mathcal{A}) = \left\{ \sum_{i=1}^na_i y_i: a_i \in \mathcal{A}, y_i \in Y, i=1,\ldots,n, n \in\mathcal{N} \right\}, \]
where $\mathcal{N}$ is the set of all natural numbers. If $X$ is a $l$-complete $\mathcal{A}$-module, then by proposition \ref{art9_utv_2_6} $(iii)$, $\mathrm{mix}(\mathrm{Lin}(Y, \mathcal{A}))$ is also a $l$-complete $\mathcal{A}$-submodule in $X$.

A set $\{x_i\}_{i \in I}$ in an $\mathcal{A}$-module $X$ is called $\mathcal{A}$-\emph{linearly independent}, if for any $a_1, \ldots, a_n \in \mathcal{A}$, $x_{i_1}, \ldots, x_{i_n} \in \{x_i\}_{i \in I}$, $n\in \mathcal{N}$, the equality $\sum\limits_{k=1}^n a_k x_{i_k} = 0$ implies equalities $a_1 = \ldots = a_n = 0$.

\begin{proposition}\label{art9_utv_3_3}
If $Y = \{x_1, \ldots, x_k\}$ is a finite $\mathcal{A}$-linearly independent subset in a $l$-complete $\mathcal{A}$-module $X$, then $\mathrm{mix}(\mathrm{Lin}(Y, \mathcal{A})) = \mathrm{Lin}(Y, \mathcal{A})$.
\end{proposition}

\begin{proof}
It is sufficient to show the \ following \ inclusion \ $\mathrm{mix}(\mathrm{Lin}(Y, \mathcal{A})) \subset \mathrm{Lin}(Y, \mathcal{A})$. \ Let \ $x \in \mathrm{mix} (\mathrm{Lin}(Y, \mathcal{A}))$, \  $\{e_i\}_{i \in I}$ \ be a partition of unity in the Boolean algebra \ $\nabla$ \ and let \ \ $\{y_i\}_{i \in I} \subset \mathrm{Lin} (Y, \mathcal{A})$ be such that $e_i x = e_i y_i$ for all $i \in I$. Since $e_i x = e_i y_i \in \mathrm{Lin}(Y, \mathcal{A})$, then $e_i x = \sum\limits_{j=1}^k a_j^{(i)} x_j$ for some $a_j^{(i)} \in \mathcal{A}$, $j=1, \ldots, k$. Hence, $e_i x = e_i(e_i x) = \sum\limits_{j=1}^k e_i a_j^{(i)} x_j$. Since $\mathcal{A}$ is a $l$-complete commutative regular algebra and $\{e_i\}_{i \in I}$ is a partition of unity in $\nabla$, then there exists a unique element $\beta_j \in \mathcal{A}$ such that $e_i \beta_j = e_i a_j^{(i)}$ for all $i\in I$, where $j \in \{1, \ldots, k\}$. Thus, $e_i x = \sum\limits_{j=1}^k e_i \beta_j x_j = e_i \left(\sum\limits_{j=1}^k \beta_j x_j\right)$ for any $i \in I$, and this implies the equality $x = \sum\limits_{j=1}^k \beta_j x_j \in \mathrm{Lin}(Y, \mathcal{A})$.

\end{proof}

We say that an $\mathcal{A}$-linearly independent system $\{x_i\}_{i \in I}$ from a $l$-complete $\mathcal{A}$-module $X$ is $\mathcal{A}$-\emph{Hamel basis}, if
\[ \mathrm{mix}(\mathrm{Lin}(\{x_i\}_{i \in I}, \mathcal{A})) = X. \]

In the case when an $\mathcal{A}$-Hamel basis is a finite set, we say that it is an $\mathcal{A}$-basis in $X$.

\begin{theorem}\label{art9_teor_4_2a}
If $\{x_i\}_{i=1}^n$, $\{y_j\}_{j=1}^k$ are $\mathcal{A}$-basises in an $\mathcal{A}$-module $X$, then $n=k$.
\end{theorem}

\begin{proof}
First we shall show the following $\mathcal{A}$-variant of one known fact from the linear algebra.

\begin{lemma}\label{art9_lemma_3_3}
Let $\{z_i\}_{i=1}^n \subset X, \{y_j\}_{j=1}^k \subset X, \{ey_j\}_{j=1}^k \subset \mathrm{Lin}(\{ez_i\}_{i=1}^n, \mathcal{A}_e)$ for nonzero $e \in \nabla$. If the set $\{e y_1, \ldots, ey_k\}$ is $\mathcal{A}_e$-linearly independent, then $k \leq n$.
\end{lemma}

\begin{proof}
We use the mathematical induction. Let us suppose that for $n=1$, $k > 1$ the equalities $e y_1 = a_1 e z_1, \ldots, e y_k = a_k e z_1$ hold, where $a_i \in \mathcal{A}_e$, $i = 1, \ldots, k$. Since $a_2 e y_1 + (-a_1) e y_2 = 0$, then $ea_1 = ea_2 = 0$, i.e. $e y_1 = e y_2 = 0$, this contradicts to $\mathcal{A}_e$-linear independence of the elements $e y_1$ and $e y_2$. Hence, $k=1$.

Now assume that the lemma holds for $n=l-1$. Let $\{z_i\}_{i=1}^l \subset X$ and the following equalities hold
\begin{equation}\label{lemma_sfdm2_8_eq_1}
ey_j = \sum_{i=1}^l a_{ji}ez_i, a_{ji} \in \mathcal{A}_e, \quad j = 1, \ldots, k, i= 1, \ldots, l.
\end{equation}

Let $a_{j_0 l}e x_l \ne 0$ for some $j_0 \in \{1, \ldots, k\}$. By reindexing $\{y_j\}_{j=1}^k$, we can assume that $a_{kl}ex_l \ne 0$, in particular $p = s(a_{kl}e) \ne 0$, wherein $p \leq e$. Since the set $\{e y_j\}_{j=1}^k$ is $\mathcal{A}_e$-linearly independent, then the set $\{p y_j\}_{j=1}^k$ is $\mathcal{A}_p$-linearly independent, wherein, by \eqref{lemma_sfdm2_8_eq_1}, we have
\begin{equation}\label{lemma_sfdm2_8_eq_2}
p y_j = \sum_{i=1}^l a_{ji}p z_i, \quad j = 1, \ldots, k.
\end{equation}

Since $\mathcal{A}$ is a regular algebra, then for the inversion $h = i(a_{kl}) \in \mathcal{A}$ the equality $ha_{kl} = s(a_{kl})$ holds. Therefore the following equality
\[ p y_k = \sum_{i=1}^{l-1} a_{ki}pz_i + a_{kl}pz_l \]
implies
\begin{equation}\label{lemma_sfdm2_8_eq_3}
pz_l = h p y_k - \sum_{i=1}^{l-1} a_{ki} h p z_i.
\end{equation}
Substitute $pz_l$ from \eqref{lemma_sfdm2_8_eq_3} in the first $(k-1)$ equalities from \eqref{lemma_sfdm2_8_eq_2} and collect similar terms, we obtain
\[ p y_j - ha_{jl}py_k = \sum_{i=1}^{l-1}\beta_{ji}pz_i \in \mathrm{Lin}(\{pz_i\}_{i=1}^{l-1}, \mathcal{A}_p) \]
for some $\beta_{ji}\in\mathcal{A}_p$, $i=1,\ldots,l-1$, $j=1,\ldots,k-1$.

Let us show that the elements $u_j = py_j - ha_{jl}py_k$, $j = 1, \ldots, k-1$ are $\mathcal{A}_p$-linearly independent. Let
\[ \sum_{j=1}^{k-1}\gamma_jpy_j - \left(\sum_{j=1}^{k-1}\gamma_j h a_{jl}\right)py_k = \sum\limits_{j=1}^{k-1}\gamma_j u_j = 0, \]
where $\gamma_j\in\mathcal{A}_p$, $j=1,\ldots,k-1$. Since $\{py_j\}_{j=1}^k$ is $\mathcal{A}_p$-linearly independent, then $p\gamma_1 = p\gamma_2 = \ldots = p\gamma_{k-1} = 0$, i.e. the set $\{u_j\}_{j=1}^{k-1}$ is $\mathcal{A}_p$-linearly independent in $pX$. By the assumption of the mathematical induction we have that $k-1 \leq l-1$, and thus $k \leq l$. The Lemma \ref{art9_lemma_3_3} is proved.
\end{proof}

Return to the proof of Theorem \ref{art9_teor_4_2a}. As $\{x_i\}_{i=1}^n$ is an $\mathcal{A}$-basis in $X$, then by Proposition \ref{art9_utv_3_3} we obtain that $X = \mathrm{Lin}(\{x_i\}_{i=1}^n, \mathcal{A})$. On the other hand, $\{y_j\}_{j=1}^k \subset X$ and $\{y_j\}_{j=1}^k$ is an $\mathcal{A}$-linearly independent set. Therefore, by Lemma \ref{art9_lemma_3_3} it follows that $k \leq n$.

Similarly, we show that $n \leq k$, and thus $n=k$.

\end{proof}

Next we need the following characterization of $\mathcal{A}$-Hamel basises.

\begin{proposition}\label{art9_utv_3_2}
For an $\mathcal{A}$-linearly independent set $\{x_i\}_{i \in I}$ in a $l$-complete $\mathcal{A}$-module $X$ the following conditions are equivalent:

(i). $\{x_i\}_{i \in I}$ is an $\mathcal{A}$-Hamel basis;

(ii). For any $x \in X$ and any nonzero idempotent $e\in\mathcal{A}$ there exists a nonzero idempotent $g \leq e$, such that $gx \in g \mathrm{Lin}(\{x_i\}_{i \in I}, \mathcal{A})$.
\end{proposition}

\begin{proof}
$(i) \Rightarrow (ii)$. If $X = \mathrm{mix}(\mathrm{Lin}(\{x_i\}_{i \in I}, \mathcal{A}))$, then for $x \in X$ there exists a partition $\{e_j\}_{j \in J}$ of unity, such that $e_j x \in e_j \mathrm{Lin}(\{x_i\}_{i \in I}, \mathcal{A})$. Since $\sup\limits_{j \in J}e_j = \mathbf{1}$, then for $0 \ne e \in \nabla$ there exists an element $j_0 \in J$ such that $g = e_{j_0}e \ne 0$, wherein $gx \in g \mathrm{Lin}(\{x_i\}_{i \in I}, \mathcal{A})$.

$(ii) \Rightarrow (i)$. Fix $0 \ne x \in X$ and for any nonzero idempotent $e\in\nabla$ choose a nonzero idempotent $g(e,x) \leq e$ such that $g(e,x) x \in g(e, x) \mathrm{Lin}(\{x_i\}_{i \in I}, \mathcal{A})$. By Theorem \ref{art9_teor_2_1}, there exists a set $\{q_j\}_{j \in J}$ of pairwise disjoint idempotents in $\mathcal{A}$ such that $\sup\limits_{j \in J}q_j = \mathbf{1}$ and $q_j x \in q_j \mathrm{Lin}(\{x_i\}_{i \in I}, \mathcal{A})$ for all $j \in J$. This means that $x \in \mathrm{mix}(\mathrm{Lin}(\{x_i\}_{i \in I}, \mathcal{A}))$, which implies the equality $X = \mathrm{mix}(\mathrm{Lin}(\{x_i\}_{i \in I}, \mathcal{A}))$.

\end{proof}

Fix some cardinal number $\gamma$. A faithful $l$-complete $\mathcal{A}$-module $X$ is called \emph{$\gamma$-homogeneous}, if there exists an $\mathcal{A}$-Hamel basis $\{x_i\}_{i \in I}$ in $X$ with $\mathrm{card}\, I = \gamma$. We say that $\mathcal{A}$-module $X$ \emph{homogeneous}, if it is a $\gamma$-homogeneous $\mathcal{A}$-module for some cardinal number $\gamma$.

If $X$ is a $\gamma$-homogeneous $\mathcal{A}$-module, then obviously, $eX$ is also $\gamma$-homogeneous $\mathcal{A}_e$-module for any nonzero idempotent $e\in\mathcal{A}$. Besides, by Proposition \ref{art9_utv_2_6} ($iv$) it follows that, if $\mathcal{A}$-module $Y$ is isomorphic to a $\gamma$-homogeneous $\mathcal{A}$-module $X$, then $Y$ is also a $\gamma$-homogeneous module.

By repeating the proof of Theorem 3.8 from \cite{lit7a}, we establish the following proposition on isomorphisms of $\gamma$-homogeneous $\mathcal{A}$-modules.

\begin{proposition}\label{art9_teor_3_13a}
If $X$ and $Y$ are $\gamma$-homogeneous $\mathcal{A}$-modules, then $X$ and $Y$ are isomorphic.
\end{proposition}

Let us give examples of $\gamma$-homogeneous $\mathcal{A}$-modules for an arbitrary cardinal number $\gamma$ and for any $l$-complete commutative regular untaly algebra $\mathcal{A}$. Consider an arbitrary set of indexes $I$ with $\mathrm{card}\,I = \gamma$. Since the algebra $\mathcal{A}$ is $l$-complete, then the Cartesian product
\[ Y = \prod\limits_{i \in I} \mathcal{A} = \{\hat{\alpha} = \{\alpha_i\}_{i \in I}: \alpha_i \in \mathcal{A}, i \in I\} \]
is a $l$-complete $\mathcal{A}$-module with coordinate-wise algebraic operations.

For any $j \in I$ consider an element $\hat{g_j} = \{g_i^{(j)}\}_{i \in I}$ from $Y$, where $g_i^{(j)} = 0$, $i \ne j$ and $g_i^{(i)} = \mathbf{1}$, $i \in I$. Clearly, that the set $\{\hat{g_j}\}_{j \in I}$ is $\mathcal{A}$-linearly independent, and, therefore, the $\mathcal{A}$-submodule $X = \mathrm{mix}\,(\mathrm{Lin}\,(\{\hat{g}_j\}_{j \in I}, \mathcal{A}))$ in $Y$ is a $\gamma$-homogeneous $\mathcal{A}$-module.

If $\gamma$ is a positive integer $n$, then for the faithful $l$-complete $\mathcal{A}$-module $Y = \prod\limits_{i=1}^n \mathcal{A} = \mathcal{A}^n$ and for $\hat{g_j} = \{g_i^{(j)}\}_{i=1}^n$, $j=1, \ldots, n$ we have that $\mathrm{Lin}\,(\{\hat{g}_j\}_{j=1}^n, \mathcal{A}) = Y$, i.e. the set $\{\hat{g}_j\}_{j=1}^n$ is an $\mathcal{A}$-Hamel basis in $Y$. Thus, Proposition \ref{art9_teor_3_13a} implies the following

\begin{corollary}\label{art9_teor_3_14}
For any positive integer $n$ there exists a unique, up to isomorphism, $n$-homogeneous $\mathcal{A}$-module, which is isomorphic to $\mathcal{A}^n$.
\end{corollary}

Let $X$ be a faithful $l$-complete $\mathcal{A}$-module, which is $\gamma$-homogeneous and $\lambda$-homogeneous simultaneously. There is a natural question, whether in this case the equality $\gamma = \lambda$ holds. Similar question was studied in classification of Kaplansky-Hilbert modules (KHM) $X$ over a commutative $AW^\ast$-algebra $\mathcal{A}$ with the Boolean algebra of projections $\nabla$ (see \cite{lit10}). In the case, when $\nabla$ is a multi-$\sigma$-finite Boolean algebra in \cite{lit10} it is proved that for a KHM $X$ the equality $\lambda = \gamma$ is always true. However, for an arbitrary complete Boolean algebra $\nabla$ this equality cannot be established. Thereby, in (\cite{lit5}, 7.4.6) the notion of \emph{strictly $\gamma$-homogeneous} KHM $X$ is defined, and this gave an opportunity to classify KHM $X$ over an arbitrary commutative $AW^\ast$-algebra $\mathcal{A}$. For the same reason, below we introduce the notion of strictly $\gamma$-homogeneous faithful $l$-complete modules over laterally complete algebras $\mathcal{A}$. With this notion we obtain necessary and sufficient conditions for $l$-complete $\mathcal{A}$-modules to be isomorphic.

Let $X$ be a faithful $l$-complete $\mathcal{A}$-module, $0 \ne e \in \nabla$. By $\varkappa(e) = \varkappa_X(e)$ we denote the smallest cardinal number $\gamma$ such that the $\mathcal{A}_e$-module $X_e$ is $\gamma$-homogeneous. If the $\mathcal{A}$-module $X$ is homogeneous, then the cardinal number $\varkappa(e)$ is defined for all nonzero $e\in\nabla$. Further, by (\cite{lit5}, 7.4.7), we assume that $\varkappa(0) = 0$.

We say that an $\mathcal{A}$-module $X$ is \emph{strictly $\gamma$-homogeneous} (compare with \cite{lit5}, 7.4.6), if $X$ is $\gamma$-homogeneous and $\gamma = \varkappa(e)$ for all nonzero $e\in\nabla$. If an $\mathcal{A}$-module $X$ is strictly $\gamma$-homogeneous for some cardinal number $\gamma$, then such $\mathcal{A}$-module $X$ is called \emph{strictly homogeneous}.

Clearly, any strictly $\gamma$-homogeneous $\mathcal{A}$-module is a $\gamma$-homogeneous $\mathcal{A}$-module. By Lemma \ref{art9_lemma_3_3} it follows that every $n$-homogeneous $\mathcal{A}$-module $X$ is a strictly $n$-homogeneous module. By Proposition \ref{art9_utv_2_6} ($iv$) every $\mathcal{A}$-module $Y$, which is isomorphic to a strictly $\gamma$-homogeneous $\mathcal{A}$-module $X$, is also strictly $\gamma$-homogeneous.

The following theorem holds.

\begin{theorem}\label{art9_teor_3_9}
Let $\lambda$ and $\gamma$ be infinite cardinal numbers and let the Boolean algebra $\nabla$ of all idempotents in a $l$-complete commutative regular algebra $\mathcal{A}$ has countable type. If a faithful $l$-complete $\mathcal{A}$-module $X$ is $\lambda$-homogeneous and $\gamma$-homogeneous simultaneously, then $\gamma = \lambda$.
\end{theorem}

\emph{Proof} of Theorem \ref{art9_teor_3_9} is similar to that of Theorem 3.4 in \cite{lit7a}.

Using Theorem \ref{art9_teor_3_9} to the $\mathcal{A}_e$-module $X_e$, we have, that Theorem \ref{art9_teor_3_9} holds in the case, when in the Boolean algebra $\nabla$ of idempotents in $\mathcal{A}$ there exists nonzero element $e$, which has a countable type. Thus, repeating the proof of Corollary 3.7 in \cite{lit7a}, we obtain the following necessary and sufficient conditions for coincidence of strictly $\gamma$-homogeneous and $\gamma$-homogeneous notions for $\mathcal{A}$-modules.

\begin{proposition}\label{art9_utv_4_8a}
Let a Boolean algebra $\nabla$ of all idempotents on a $l$-complete commutative regular algebra $\mathcal{A}$ be multi-$\sigma$-finite. If $\gamma$ is an infinite cardinal number and $X$ is a $\gamma$-homogeneous $\mathcal{A}$-module, then the module $X$ is strictly $\gamma$-homogeneous.
\end{proposition}

\begin{sloppypar}
The following proposition enables to ``glue'' $\gamma$-homogeneous (strictly $\gamma$-homogeneous) $\mathcal{A}$-modules.
\end{sloppypar}

\begin{proposition}\label{art9_utv_4_9a}
Let $\mathcal{A}$ be a $l$-complete commutative regular algebra, let $X$ be a $l$-complete $\mathcal{A}$-module and let $\{e_i\}_{i \in I}$ be a set of pairwise disjoint nonzero idempotents in $\mathcal{A}$ and $e=\sup\limits_{i \in I}e_i$. If $ X_{e_i}$ is a $\gamma$-homogeneous (respectively, strictly $\gamma$-homogeneous) $\mathcal{A}_{e_i}$-module for all $i \in I$, then the $\mathcal{A}_e$-module $X_e$ is also $\gamma$-homogeneous (respectively, strictly $\gamma$-homogeneous).
\end{proposition}

\emph{Proof} is similar to that of Proposition 3.10 in \cite{lit7a}.

\section{Classification of faithful $l$-complete $\mathcal{A}$-modules}
\label{sec:4}
In this section it is proved that every faithful laterally complete $\mathcal{A}$-module is isomorphic to a Cartesian product of strictly homogeneous $\mathcal{A}$-modules. The important step in obtaining such an isomorphism is the following theorem.

\begin{theorem}\label{art9_teor_5_1}
Let $\mathcal{A}$ be a $l$-complete commutative regular algebra, let $\nabla$ be a Boolean algebra of all idempotents in $\mathcal{A}$ and let $X$ be a faithful $l$-complete $\mathcal{A}$-module. Then there exists a nonzero idempotent $p\in\nabla$ such that $X_p$ is a strictly homogeneous $\mathcal{A}_p$-module.
\end{theorem}

\begin{proof}
Using Proposition \ref{art9_utv_2_7} ($i$), we choose $x_0 \in X$ such that $s(x_0) = \mathbf{1}$. If $X = \mathrm{Lin}(x_0, \mathcal{A})$, then $X$ is a strictly $1$-homogeneous module and Theorem \ref{art9_teor_5_1} is proved.

Assume that $X \ne \mathrm{mix}\,(\{x_0\})$. We consider in $X$ the following nonempty family of subsets
\[ \mathscr{E} = \{B \subset X: x_0 \in B, B - \mathcal{A}\text{-linearly independent set}\}. \]
We introduce in $\mathscr{E}$ a partial order by $B \leq C \Leftrightarrow B \subset C$. By Zorn's lemma there exists maximal element $D$ in $\mathscr{E}$. If $D$ is an $\mathcal{A}$-Hamel basis in $X$, then $X$ is  $(\mathrm{card}\, D)$-homogeneous $\mathcal{A}$-module.

Assume that $X \ne \mathrm{mix}\,(\mathrm{Lin}\,(D, \mathcal{A}))$. If for any nonzero $e\in\nabla$ there exists $0 \ne q_e \in \nabla$ such that $q_e\, \mathrm{mix}\,(\mathrm{Lin}\,(D, \mathcal{A})) = q_e X$, then from Proposition \ref{art9_utv_2_6} ($iii$) and Proposition \ref{art9_utv_2_7} ($ii$) it follows that $X = \mathrm{mix}\,(\mathrm{Lin}\,(D, \mathcal{A}))$, which contradicts our assumption. Hence, there exists nonzero $e\in\nabla$ such that the following condition holds:
\[ g\, \mathrm{mix}\,(\mathrm{Lin}\,(D, \mathcal{A})) \ne gX \ \hbox{for all non zero} \ g \in \nabla_e. \tag{1} \]

Denote by $\mathscr{L}$ a set of all nonzero $e\in\nabla$ with property $(1)$. Put $e_0 = \sup\mathscr{L}$ and show that the equality $e_0 = \mathbf{1}$ fails.

Assume that $e_0 = \mathbf{1}$. In this case for every nonzero $q \in \nabla$ there exists $e\in\mathscr{L}$ such that $g = qe \ne 0$. Hence, $gX \ne g\mathrm{mix}\,(\mathrm{Lin}\,(D, \mathcal{A}))$ (see (1)), which implies
\[ q X \ne q\, \mathrm{mix}\,(\mathrm{Lin}\,(D, \mathcal{A})). \tag{2} \]

Show that for any nonzero $q\in\nabla$ there exists a nonzero idempotent $r \leq q$ such that for any $0 \ne g \in \nabla_r$ the following property holds:
\[ \hbox{There exists } x_g \in gX \hbox{ such that } s(x_g)=g \hbox{ and } \ lx_g \not\in \mathrm{Lin}\,(D, \mathcal{A}) \hbox{ for all } 0 \ne l \in \nabla_g. \tag{3} \]

If this is not true, then there exists a nonzero $q \in\nabla$ such that for every $0 \ne r \in \nabla_q$ there exists a nonzero idempotent $g_r \in \nabla_r$ without property $(3)$, i.e. for any $x \in g_r X$ with $s(x) = g_r$ there exists a nonzero idempotent \ $e(x_g,r) \leq g_r \leq q$ \ such that
$$
e(x_g,r) x \in e(x_g,r) \mathrm{Lin}\,(D, \mathcal{A}) \subset \mathrm{Lin}\,(D, \mathcal{A}).
$$

\begin{sloppypar}
Show that, in this case, $g_q X = g_q \mathrm{mix}\,(\mathrm{Lin}\,(D, \mathcal{A}))$. Let $x$ be a nonzero element in $g_q X$, in particular, $0 \ne s(x) \leq g_q$. For any nonzero idempotent $a \leq s(x)$ there exists a nonzero idempotent $e(ax,a) \leq a$ such that $e(ax,a) x \in \mathrm{Lin}\,(D,\mathcal{A})$. By Theorem \ref{art9_teor_2_1}, there exists a partition $\{e_i\}_{i \in I}$ of support $s(x)$ such that $e_i x \in s(x)\mathrm{Lin}\,(D, \mathcal{A})$ for all $i \in I$. This means that $x \in \mathrm{mix}\,(s(x)\mathrm{Lin}\,(D, \mathcal{A})) = s(x)\mathrm{mix}\,(\mathrm{Lin}\,(D, \mathcal{A}))$ (see Proposition \ref{art9_utv_2_6} $(ii)$). Since $s(x) \leq g_q$, we have that $x \in g_q \mathrm{mix}\,(\mathrm{Lin}\,(D, \mathcal{A}))$, which implies the inclusion $g_q X \subset g_q \mathrm{mix}\,(\mathrm{Lin}\,(D, \mathcal{A}))$. On the other hand, by $l$-completeness of an $\mathcal{A}_{g_q}$-module $g_q X$ we have that
\[g_q \mathrm{mix}\,(\mathrm{Lin}\,(D, \mathcal{A})) \subset g_q \mathrm{mix}\,(X) = \mathrm{mix}\,(g_q X) = g_q X. \]
Hence, $g_q X = g_q \mathrm{mix}\,(\mathrm{Lin}\,(D, \mathcal{A}))$, which contradicts to (2).
\end{sloppypar}

Thus, for every nonzero $q\in\nabla$ there exists a nonzero idempotent $r \leq q$ such that for any $0 \ne g \in \nabla_r$ property $(3)$ holds.

Again by Theorem \ref{art9_teor_2_1}, we choose a partition $\{g_j\}_{j \in J}$ of the idempotent $r$ and a set $\{x_{g_j}\}_{j \in J}$ in $rX$, such that $s(x_{g_j}) = g_j$ and $l x_{g_j} \not\in \mathrm{Lin}\,(D, \mathcal{A})$ for all $0 \ne l \in \nabla_{g_i}$.

Since $rX$ is a $l$-complete $\mathcal{A}_r$-module, then there exists $x \in rX$ such that $g_j x = x_{g_j}$. In particular, $s(x) = r$, wherein $lx \not\in \mathrm{Lin}\,(D, \mathcal{A})$ for all $0 \ne l \in \nabla_r$.

Again by Theorem \ref{art9_teor_2_1} we choose a partition $\{r_k\}_{k \in K}$ of the unity $\mathbf{1}$ in the Boolean algebra $\nabla$ and a set $\{x_k\}_{k \in K}$ in $X$, such that $s(x_k) = r_k$ and $l x_k \not\in \mathrm{Lin}\,(D, \mathcal{A})$ for any $0 \ne l \in \nabla_{r_k}$. By $l$-completeness of the $\mathcal{A}$-module $X$ there exists $\hat{x} \in X$ such that $r_k \hat{x} = x_k$ for all $k \in K$. In this case $s(\hat{x}) = \mathbf{1}$ and $l \hat{x} \not\in \mathrm{Lin}\,(D, \mathcal{A})$ for any $0 \ne l \in \nabla$.

Show that the set $D \cup \{\hat{x}\}$ is $\mathcal{A}$-linearly independent. Let $a_0 \hat{x} + \sum\limits_{i=1}^n a_i x_i = 0$, where $a_0, a_i \in \mathcal{A}$, $x_i \in D$, $i=1, \ldots, n$. If $a_0 = 0$, then $\sum\limits_{i=1}^n a_i x_i = 0$ and by $\mathcal{A}$-linear independence of the set $D$ it follows that $a_i = 0$ for all $i=1, \ldots, n$. If $a_0 \ne 0$, then $s(a_0) \ne 0$ and for $i(a_0)=h \in \mathcal{A}$ we have that $h a_0 = s(a_0)$ and $s(a_0)\hat{x} = - \sum\limits_{i=1}^n a_i h x_i \in \mathrm{Lin}\,(D, \mathcal{A})$, which is not true. Hence, the set $D \cup \{\hat{x}\}$ is $\mathcal{A}$-linearly independent in $X$, which contradicts to maximality of the set $D$.

Thus the equality $e_0 = \mathbf{1}$ is impossible. This means that $e = \mathbf{1} - e_0 \ne 0$. By construction of the idempotent $e_0$, every nonzero idempotent $r \leq e$ does not have property $(1)$. Hence, for any $0 \ne r \in \nabla_e$ there exists a nonzero idempotent $p_r \leq r$ such that
\[ p_r X = p_r \mathrm{mix}\,(\mathrm{Lin}\,(D, \mathcal{A})) = \mathrm{mix}\,(\mathrm{Lin}\,(p_r D, \mathcal{A}_{p_r }))= p_r \mathrm{mix}\,(\mathrm{Lin}\,(e D, \mathcal{A}_{e})). \]
From Propositions \ref{art9_utv_2_6} $(iii)$ and \ref{art9_utv_2_7} $(ii)$ it follows that
\[ eX = \mathrm{mix}\,(\mathrm{Lin}\,(eD, \mathcal{A}_e)). \]
Since $eD$ is an $\mathcal{A}_e$-linearly independent subset in the $\mathcal{A}_e$-module $eX$, then $eD$ is an $\mathcal{A}_e$-basis in $eX$, i.e. $eX$ is a $\gamma$-homogeneous $\mathcal{A}_e$-module, where $\gamma = \mathrm{card}\,(eD)$. In particular, a cardinal number $\varkappa(p)$ is defined for all nonzero $p \in \nabla_e$. Let $\gamma_e$ be the smallest cardinal number in the set of cardinal numbers $\{\varkappa(p): 0 \ne p \leq e\}$, i.e. $\gamma_e = \varkappa(p)$ for some nonzero $p \leq e$. By the choice of the idempotent $p$ it follows that $\gamma_e = \varkappa(p) = \varkappa(q)$ for all $0 \ne q \in \nabla_p$. This means that the $\mathcal{A}_p$-module $X_p$ is strictly homogeneous.

\end{proof}

Now everything is ready to obtain the isomorphism from the faithful laterally complete $\mathcal{A}$-module to the Cartesian product of strictly homogeneous $\mathcal{A}$-modules.

\begin{theorem}\label{art9_teor_5_2}
Let $\mathcal{A}$ be a $l$-complete commutative regular algebra, let $\nabla$ be a Boolean algebra of all idempotents in $\mathcal{A}$ and let $X$ be a faithful $l$-complete $\mathcal{A}$-module. Then there exist a uniquely defined set of pairwise disjoint nonzero idempotents $\{e_i\}_{i \in I} \subset \nabla$ and a set of pairwise different cardinal numbers $\{\gamma_i\}_{i \in I}$ such that $\sup\limits_{i \in I}e_i = \mathbf{1}$ and $X_{e_i}$ is a strictly $\gamma_i$-homogeneous $\mathcal{A}_{e_i}$-module for all $i \in I$. In this case, the $\mathcal{A}$-modules $X$ and $\prod\limits_{i \in I} X_{e_i}$ are isomorphic.
\end{theorem}

\begin{proof}
By Theorem \ref{art9_teor_5_1} for every nonzero idempotent $e\in\mathcal{A}$ there exists a nonzero idempotent $g \leq e$ such that $X_g$ is a strictly homogeneous $\mathcal{A}_g$-module. By Theorem \ref{art9_teor_2_1}, choose a set of pairwise disjoint nonzero idempotents $\{q_j\}_{j \in J}$ such that $\sup\limits_{j \in J}q_j = \mathbf{1}$ and $q_j X$ is a strictly $\lambda_j$-homogeneous $\mathcal{A}_{q_j}$-module for all $j \in J$. We decompose the set of cardinal numbers $A = \{\lambda_j\}_{j \in J}$ as a union of disjoint subsets $A_i$ in such a way that every $A_i$ consists of equal cardinal numbers from $A$. By $\gamma_i$ denote an element in $A_i$. By Proposition \ref{art9_utv_4_9a}, for $e_i = \sup\{q_j: \lambda_j \in A_i\}$ we have that the $\mathcal{A}_{e_i}$-module $X_{e_i}$ is strictly $\gamma_i$-homogeneous. Moreover, by Proposition \ref{art9_utv_2_8}, the $\mathcal{A}$-module $X$ and $\prod\limits_{i \in I}e_i X$ are isomorphic.

Assume, that there exist other sets of pairwise disjoint nonzero idempotents $\{g_j\}_{j \in J}$ and pairwise different cardinal numbers $\{\mu_j\}_{j \in J}$, such that $\sup\limits_{j \in J}g_j = \mathbf{1}$ and $X_{g_j}$ is a strictly $\mu_j$-homogeneous $\mathcal{A}_{g_j}$-module for all $j \in J$. For any fixed $j \in J$, by the equality $\sup\limits_{i \in I}e_i = \mathbf{1}$, we have that $g_j = \sup\limits_{i \in I}e_i g_j$. If there exist two different indexes $i_1, i_2 \in I$ such that $e_{i_1} g_j \ne 0$ and $e_{i_2} p_j \ne 0$, then
\[ \mu_j = \varkappa(g_j) = \varkappa(e_{i_1}g_j) = \varkappa(e_{i_1}) = \gamma_{i_1} \ne \gamma_{i_2} = \varkappa(e_{i_2}) = \varkappa(e_{i_2}g_j) = \mu_j. \]
By this contradiction, it follows that $e_i g_j = 0$ for all $i \in I$ except one index, which we denote by $i(j)$. Since $e_{i(j)} g_j \neq 0$, we have that
\[ \mu_j = \varkappa(g_j) = \varkappa(e_{i(j)} g_j) = \varkappa(e_{i(j)}) = \gamma_{i(j)}. \]
If $g_j \ne e_{i(j)}$, then by the equality $\sup\limits_{j \in J}g_j = \mathbf{1}$, there exists index $j_1 \in J$, $j_1 \ne j$ such that $e_{i(j)}g_{j_1} \ne 0$. Hence,
\[ \mu_j = \gamma_{i(j)} = \varkappa(e_{i(j)}) = \varkappa(e_{i(j)} g_{j_1}) = \varkappa(g_{j_1}) = \mu_{j_1}, \]
which is not true. Thus, $g_j = e_{i(j)}$ and $\mu_j = \gamma_{i(j)}$.

For the same reason, for any $i \in I$ there exists the unique index $j(i)$ such that $e_i = g_{j(i)}$ and $\gamma_i = \mu_{j(i)}$.

\end{proof}

The partition $\{e_i\}_{i \in I}$ of unity in a Boolean algebra of idempotents in $\mathcal{A}$ and the set of cardinal numbers $\{\gamma_i\}_{i \in I}$ in Theorem \ref{art9_teor_5_2} are called a passport for a faithful laterally complete $\mathcal{A}$-module $X$ and denoted by $\Gamma(X) = \{(e_i(X), \gamma_i(X))\}_{i \in I(X)}$.

Thus, a passport $\Gamma(X) = \{(e_i(X), \gamma_i(X))\}_{i \in I(X)}$ for a faithful $l$-complete $\mathcal{A}$-module $X$ means that $X = \prod\limits_{i \in I(X)}e_i(X) X$ (up to an isomorphism), where $e_i(X) X$ is a strictly $\gamma_i(X)$-homogeneous $\mathcal{A}_{e_i}$-module for all $i \in I(X)$, $e_i(X) \ne 0$, $e_i(X) e_j(X) = 0$, $\gamma_i(X) \ne \gamma_j(X)$, $i \ne j$, $i,j \in I(X)$, $\sup\limits_{i \in I(X)}e_i(X) = \mathbf{1}$.

The following theorem gives a criterion for isomorphism between faithful $l$-complete $\mathcal{A}$-modules, by using the notion of passport for these $\mathcal{A}$-modules.

\begin{theorem}\label{art9_teor_5_3}
Let $\mathcal{A}$ be a $l$-complete commutative regular algebra, $X$ and $Y$ be a faithful $l$-complete $\mathcal{A}$-modules. The following conditions are equivalent:

(i) $\Gamma(X) = \Gamma(Y)$;

(ii) $\mathcal{A}$-modules $X$ and $Y$ are isomorphic.
\end{theorem}

\begin{proof}
\begin{sloppypar}
$(i) \Rightarrow (ii)$. Let $\{(e_i(X), \gamma_i(X))\}_{i \in I(X)} = \Gamma(X) = \Gamma(Y) = \{(e_i(Y), \gamma_i(Y))\}_{i \in I(Y)}$, i.e. $I(X) = I(Y) := I$, $e_i(X) = e_i(Y) := e_i$ and $\gamma_i(X) = \gamma_i(Y) := \gamma_i$ for all $i \in I$. By Theorem \ref{art9_teor_5_2}, there exists an isomorphism $U$ from $\mathcal{A}$-module $X$ onto $\mathcal{A}$-module $\prod\limits_{i \in I}e_i X$ (respectively an isomorphism $V$ from $\mathcal{A}$-module $Y$ onto $\mathcal{A}$-module $\prod\limits_{i \in I}e_i Y$), where $U(x) = \{e_i x\}_{i \in I}$ (respectively, $V(y) = \{e_i y\}_{i \in I}$) for every $x \in X$ (respectively, for every $y \in Y$).

Since $e_i X$ (respectively, $e_i Y$) is a strictly $\gamma_i$-homogeneous $\mathcal{A}_{e_i}$-module, then by Proposition \ref{art9_teor_3_13a}, for all $i \in I$ there exists an isomorphism $U_i$ from the $\mathcal{A}_{e_i}$-module $e_i X$ onto the $\mathcal{A}_{e_i}$-module $e_i Y$. It is clear that a map $\Phi: X \rightarrow Y$, defined by the equality
\[ \Phi(x) = V^{-1}(\{U_i(e_i x)\}_{i \in I}). \]
is an isomorphism from the $\mathcal{A}$-module $X$ onto the $\mathcal{A}$-module $Y$.

$(ii) \Rightarrow (i)$. Let $\Psi$ be an isomorphism from $X$ onto $Y$ and $\Gamma(X) = \{(e_i(X), \gamma_i(X))\}_{i \in I(X)}$ be a passport for a faithful $l$-complete $\mathcal{A}$-module $X$. By Proposition \ref{art9_utv_2_6} $(iv)$, the following $\mathcal{A}_{e_i(X)}$-module
\[ Y_i = \Psi(e_i(X)X) = e_i(X) \Psi(X) = e_i(X) Y \]
is strictly $\gamma_i(X)$-homogeneous. This means that $\{(e_i(X), \gamma_i(X))\}_{i \in I}$ is a passport for the faithful $l$-complete $\mathcal{A}$-module $Y$, i.e. $\Gamma(X) = \Gamma(Y)$.

\end{sloppypar}
\end{proof}

Let $\mathcal{A}$ be a $l$-complete commutative regular algebra, let $\nabla$ be a Boolean algebra of all idempotents in $\mathcal{A}$. A faithful $l$-complete $\mathcal{A}$-module $X$ is called finitely-dimensional, if there exist a finite partition $\{e_i\}_{i=1}^k$ of unity in the Boolean algebra $\nabla$ ($e_i \neq 0, i=1, \ldots, k)$ and a finite set $\{n_i\}_{i=1}^k$ of natural numbers ($n_1 < n_2 < \ldots < n_k$) such that $X_{e_i}$ is an $n_i$-homogeneous $\mathcal{A}_{e_i}$-module for all $i = 1, \ldots, k$.

This means that any finitely-dimensional $\mathcal{A}$-module $X$ has a passport of the following form
\[ \Gamma(X) = \{(e_i(X), n_i(X))\}_{i=1}^k, \]
where
\[ e_1(X) + \ldots + e_k(X) = \mathbf{1}, n_1(X) < \ldots < n_k(X) < \infty. \]

\begin{theorem}\label{art9_teor_5_4}
For a faithful $l$-complete $\mathcal{A}$-module $X$ the following conditions are equivalent:

(i). $X$ is a finitely-dimensional module;

(ii). $X$ is a finitely-generated module, i.e. there exists a finite set $\{x_i\}_{i=1}^m$ of elements in $X$ such that $X = \mathrm{Lin}(\{x_i\}_{i=1}^m, \mathcal{A})$;

(iii). There exists a positive integer $m$ such that for any nonzero idempotent $e\in\mathcal{A}$ any $\mathcal{A}_e$-linearly independent set in $X_e$ consists of not more than $m$ elements.
\end{theorem}

\begin{proof}
$(i)\Rightarrow(ii)$. Let $\Gamma(X) = \{(e_i(X), n_i(X))\}_{i=1}^k$ be a passport for the $\mathcal{A}$-module $X$. For every $i = 1, \ldots, k$ we choose the $\mathcal{A}_{e_i}$-basis $\{x_j^{(i)}\}_{j=1}^{n_i}$ in $X_{e_i}$. If $x \in X$, then $e_i x = \sum\limits_{j=1}^{n_i} a_j^{(i)}x_j^{(i)}$, where $a_j^{(i)} \in \mathcal{A}_{e_i}$. Hence,
\[ x = \sum\limits_{i=1}^k e_i x = \sum\limits_{i=1}^k\sum\limits_{j=1}^{n_i} a_j^{(i)}g_j^{(i)} \in \mathrm{Lin}\ (\{x_j^{(i)}\}_{j=\overline{1,n_i}, i=\overline{1,k}}, \mathcal{A}). \]
This means that $\mathcal{A}$-module $X$ is finitely-generated.

$(ii)\Rightarrow(iii)$. If $X = \mathrm{Lin}(\{x_i\}_{i=1}^m, \mathcal{A})$, $e$ is a nonzero idempotent in $\mathcal{A}$ and $\{y_j\}_{j=1}^l$ is an $\mathcal{A}_e$-linearly independent set in $X_e$, then by Lemma \ref{art9_lemma_3_3}, it follows that $l \leq m$.

$(iii)\Rightarrow(i)$. By Theorem \ref{art9_teor_5_2}, there exist a set of pairwise disjoint nonzero idempotents $\{e_i\}_{i \in I}$ and a set of pairwise different cardinal numbers $\{\gamma_i\}_{i \in I}$ such that $\sup\limits_{i \in I}e_i = \mathbf{1}$ and $X_{e_i}$ is a strictly $\gamma_i$-homogeneous $\mathcal{A}_{e_i}$-module for all $i \in I$. If $\gamma_i > m$, then in $X_{e_i}$ there exists a finite set $\{x_i\}_{i=1}^l$, which consist of $\mathcal{A}_e$-linearly independent elements, and besides $l > m$, which contradicts to condition $(iii)$. Hence, $\gamma_i \leq m$ for all $i \in I$. Since natural numbers $\{\gamma_i\}_{i \in I}$ are pairwise different, then $I$ is a finite set, i.e. $\{\gamma_i\}_{i \in I} = \{n_i\}_{i=1}^k$, where $n_1 < n_2 < \ldots n_k$. Hence, the $\mathcal{A}$-module $X$ is finitely-dimensional.

\end{proof}

The following description of finitely-dimensional $\mathcal{A}$-modules follows directly from Theorem \ref{art9_teor_5_2} and Corollary \ref{art9_teor_3_14}.

\begin{corollary}\label{art9_teor_4_5}
If $X$ is a finitely-dimensional $\mathcal{A}$-module, then there exist an uniquely defined finite partition $\{e_i\}_{i=1}^k$ of unity in the Boolean algebra of all idempotents in $\mathcal{A}$ and a finite set of positive integers $n_1 < \ldots < n_k$ such that the $\mathcal{A}$-module $X$ is isomorphic to the $\mathcal{A}$-module $\prod\limits_{i=1}^k \mathcal{A}^{n_i}_{e_i}$ (here $e_i \ne 0$ for all $i = 1, \ldots, k$).
\end{corollary}

A faithful $l$-complete $\mathcal{A}$-module $X$ is called $\sigma$-finitely-dimensional, if there exist a countable partition $\{e_i\}_{i=1}^\infty$ of unity in the Boolean algebra of all idempotents in $\mathcal{A}$ ($e_i \neq 0, i=1,2,\ldots$) and a countable set $\{n_i\}_{i=1}^\infty$ of positive integers ($n_1 < n_2 < \ldots $) such that $X_{e_i}$ is an $n_i$-homogeneous $\mathcal{A}_{e_i}$-module for all $i = 1, 2, \ldots$

By Theorem \ref{art9_teor_5_2} and Corollary \ref{art9_teor_3_14} we obtain the following description of $\sigma$-finitely-dimensional $\mathcal{A}$-modules.

\begin{corollary}\label{art9_teor_4_6}
If $X$ is a $\sigma$-finitely-dimensional $\mathcal{A}$-module, then there exist a uniquely defined countable partition $\{e_i\}_{i=1}^\infty$ of unity in the Boolean algebra of all idempotents in $\mathcal{A}$ and a countable set of positive integers $n_1 < n_2 < \ldots $ such that the $\mathcal{A}$-module $X$ is isomorphic to the $\mathcal{A}$-module $\prod\limits_{i=1}^\infty \mathcal{A}^{n_i}_{e_i}$ (here $e_i \ne 0$ for all $i = 1, 2, \ldots$).
\end{corollary}

National University of Uzbekistan, 100174, Vuzgorodok, Tashkent, Uzbekistan;

e-mail: chilin@ucd.uz, vladimirchil@gmail.com;
        
        karimovja@mail.ru

\end{document}